\documentclass[10pt]{article}
\usepackage{hyperref}
\usepackage{amsmath,amssymb}
\input{liemacs10.sty}

\usepackage{amsmath}
\usepackage{amssymb,latexsym,amstext,amsthm}
\usepackage{verbatim} 
\usepackage{eucal} 
\usepackage{color}
\date{}

\def\a{\alpha}

\def\ra{\rangle}
\def\la{\langle}

\renewcommand{\mlabel}{\label}
\newcommand{\oo}{{\bar{1}}}
\renewcommand{\oe}{{\bar{0}}}

\newcommand{\psu}{\mathop{{\mathfrak{psu}}}\nolimits}
\newcommand{\osp}{\mathop{{\mathfrak{osp}}}\nolimits}
\renewcommand{\str}{\mathop{\mathrm{str}}\nolimits}

\newcommand{\cheis}{\mathfrak{ch}}
\newcommand{\pg}{\mathfrak{pg}}
\newcommand{\psl}{\mathfrak{psl}}
\newcommand{\pq}{\mathfrak{pq}}

\numberwithin{equation}{section}

\def\qed{\hfill$\Box$}

\def\aij{[\fa_i,\fa_j]}

\begin{document}


\title{Finite dimensional compact and\\ unitary Lie superalgebras}
\author{Saeid Azam$^\text{a,}$\thanks{Corresponding author at: Department of Mathematics, University of Isfahan, P.O.Box 81745-163, Isfahan, Iran.}\;,
Karl-Hermann Neeb$^\text{b}$
}
\maketitle
\noindent$^\text{a}$
{\it \small{School of Mathematics, Institute for Research in Fundamental Sciences (IPM), P.O.Box: 19395-5746, Tehran, Iran, and Department of Mathematics, University of Isfahan,
P.O.Box 81745-163, Isfahan, Iran.}}

\noindent$^\text{b}$ {\it \small{Department of Mathematics, Friedrich-Alexander University Erlangen-Nuremberg,
Cauerstrasse 11, 91058, Erlangen, Germany.}}

{\let\thefootnote\relax\footnotetext{E-mail addresses: azam@sci.ui.ac.ir (S. Azam), neeb@math.fau.de (K.-H. Neeb).}}

\begin{abstract}
{Motivated by the theory of unitary representations
of finite dimensional Lie supergroups,
we describe those Lie superalgebras
which have a faithful finite dimensional unitary representation.
We call these Lie superalgebras unitary.
This is achieved by describing the classification of real finite dimensional compact simple
Lie superalgebras, and analyzing, in a rather elementary and direct way,
the decomposition of reductive Lie superalgebras
($\g$ is a semisimple $\g_\oe$-module) over fields of characteristic zero into ideals.
}
\end{abstract}

\vspace{2mm}
\section{Introduction}
\mlabel{sec:0}

A real Lie algebra $\fk$ is called {\it compact} if it is the
Lie algebra of a compact Lie group (\cite[Sect.~12.1]{HN12}).
There are many characterizations
of compact Lie algebras:

\begin{thm} \mlabel{thm:0.1} For a finite dimensional real Lie algebra $\fk$, the
following are equivalent:
\begin{itemize}
\item[\rm(i)] $\fk$ is compact.
\item[\rm(ii)] $\fk$ carries a positive definite invariant symmetric bilinear
form.
\item[\rm(iii)] $\fk$ has a faithful finite dimensional unitary representation,
i.e., it is isomorphic to a subalgebra of $\fu(n,\C)$ for some $n \in \N$.
\end{itemize}
\end{thm}

In this paper we discuss variants of compactness for finite dimensional
real Lie superalgebras.

\begin{defn} (a) A Lie superalgebra
$\g = \g_\oe \oplus \g_\oo$ over a field $\K$ of characteristic zero
is called {\it reductive}
if $\g$ is a semisimple module for the Lie algebra $\g_\oe$
with respect to the adjoint representation.
It is called {\it semisimple} if it contains no non-zero  solvable ideal.

A simple reductive Lie superalgebra
is called {\it classical} (cf.~\cite{Ka77}). For a Lie superalgebra $\g$
we always assume $\g_\oo \not=\{0\}$, unless otherwise stated.

(b) A real Lie superalgebra $\g$ is called {\it compact} if
$\g_\oe$ is a {\it compactly embedded subalgebra}, i.e.,
the subgroup of $\Aut(\g)$ generated by $e^{\ad \g_\oe}$ has compact closure.
Since this implies that $\g$ is a semisimple $\g_\oe$-module, compact
Lie superalgebras are reductive.

(c) A real Lie superalgebra $\g$ is called {\it unitary} if it has a faithful
unitary representation, i.e., if it is isomorphic to a subalgebra
of some Lie superalgebra $\fu(p|q;\C)$.
\end{defn}

If $\g$ is a real Lie algebra, i.e., $\g_\oo =\{0\}$, then compactness
and unitarity are the same by Theorem~\ref{thm:0.1}.
For Lie superalgebras it is easy to see that unitarity implies compactness,
but that the converse is false in general. The main goal of this paper is
to describe the structure of unitary Lie superalgebras. This is in particular
motivated by the theory of unitary representations
of Lie supergroups (\cite{CCTV06},  \cite{NS11}), for which a thorough understanding
of the finite dimensional unitary Lie superalgebras is a fundamental ingredient.

In our structural analysis of unitary Lie superalgebras we proceed as follows.
In Section~\ref{sec:1} we describe the classification of compact simple
Lie superalgebras (Theorem~\ref{thm:1.4}) and discuss some of their relatives.
There are only $4$ families:
\[ \su(n|m;\C), n > m, \quad
\psu(n|n;\C), n \geq 2, \quad
\pq(n), n > 2, \quad \mbox{ and } \quad
\fc(n), n \geq 2.\]
This classification result is derived from Parker's classification
of real forms of classical Lie superalgebras (\cite{Pa80})
by inspecting which of these real forms are compact.
In Section~\ref{sec:2} we take a first closer look at special properties
of unitary Lie superalgebras. A crucial property of a unitary Lie superalgebra
is the existence of linear functionals $\omega$ on $\g_\oe$ for which the
symmetric bilinear form $(X,Y) \mapsto \omega([X,Y])$ on $\g_\oo$ is positive
definite. This implies in particular that the closed convex cone
$\cC(\g) \subeq \g_\oe$ generated by the elements $[X,X]$, $X \in \g_\oo$,
is pointed and that $\g_\oe$ has non-trivial center (cf.~\cite{NS11}). Although these
requirements fail for some simple compact Lie superalgebras,
they always possess unitary central extensions.

Section~\ref{sec:3} provides the main information on the decomposition
of reductive Lie superalgebras into ideals. Since this analysis goes through without additional problem
over fields of characteristic zero, it is carried out in this context.
Here we focus on the essential case where $\g$ is a reductive Lie superalgebra
satisfying $\g_\oe = [\g_\oo,\g_\oo]$, i.e.,
where $\g$ is generated by its odd part $\g_\oo$. Our main result is the Structure
Theorem~\ref{thm:3.8} which describes the commutator algebra
$[\g,\g]$ as a quotient of a direct sum of the center and
of Lie superalgebras $\g(j)$ and $\fc(k)$ by a central subspace. Here
$\g(j)/\z(\g(j))$ is a classical Lie superalgebra and
$\fc(k)/\fz(\fc(k)) \cong \fk \otimes\Lambda_1$, where
$\fk$ is a simple Lie algebra and $\Lambda_n$ denotes the Gra\ss{}mann algebra
with $n$ generators. The results in this section are closely related to
Elduque's results on reductive Lie superalgebras over fields of characteristic
zero (\cite{El96}; see also \cite{BR78}). In principle we could have derived our results from his, but
our approach is rather elementary, very direct and never requires algebraic closedness
of the base field. In particular it works over
$\R$.
\footnote{For an inductive classification of reductive Lie superalgebras equipped with
an even non-degenerate supersymmetric
$\g$-invariant bilinear {form} see \cite{Be00}.}

In the final Section~\ref{sec:4} everything is put together to
obtain a description of the structure of unitary Lie superalgebras.
Here one of the main results is Theorem~\ref{thm:3.7} which asserts
that if $\g$ is unitary, then the ideals $\g(j)$ are isomorphic to
\[ \su(n|m;\C), \ n \geq m \geq 1, \quad
\fq(n), \ n\geq 2, \quad \mbox{  or } \quad \fc(n), \ n \geq 2.\]
Here the Lie superalgebras of the form $\su(n|m;\C)$, $n > m$, and $\fc(n)$ must be
direct summands, whereas the Lie superalgebras
$\fq(n)$ and $\su(n|n;\C)$ intersect the center non-trivially.
The ideals $\fc(k)$ are direct sums of a central ideal
and a one-dimensional central extension of $\fk \otimes \Lambda_1$, where
$\fk$ is a compact simple Lie algebra. Finally, we show that all these types
occur.

\tableofcontents

\section{Compact classical Lie superalgebras}
\mlabel{sec:1}

In this section we start our analysis
with the classification of compact simple
Lie superalgebras (Theorem~\ref{thm:1.4}).
In passing, we introduce some related extensions
and discuss their matrix realizations. {We recall that a real Lie superalgebra $\g$ is called  compact if
$\g_\oe$ is a compactly embedded subalgebra, i.e.,
the subgroup of $\Aut(\g)$ generated by $e^{\ad \g_\oe}$ has compact closure.}

According to \cite{Ka77}, the classical simple complex Lie superalgebras are
grouped in $6$ infinite series and $3$ exceptional types.
\begin{itemize}
\item  $\begin{cases}
A(n,m) & := \fsl(n+1|m+1,\C)  \text{ for } n > m \geq 0 \\
A(n,n) &:= \psl(n+1|n+1;\C) := \fsl(n+1|n+1,\C)/\C \1 \text{ for } n > 0.
  \end{cases}$
\item $B(n,m) := \osp(2n+1| 2m)$, $n \geq 0$, $m > 0$.
\item $C(n) := \osp(2|2n-2)$, $n \geq 2$.
\item $D(n,m) := \osp(2n|2m)$, $n \geq 2$, $m > 0$.
\item $Q(n) := \tilde Q(n)/\C \1$, $\tilde Q(n) := \Big\{\pmat{ a & b \\ b & a} \in
\fsl(n+1|n+1,\C) \: \tr b =0\Big\}$,
for $n \geq 2$ (the {\it queer Lie superalgebra}).
\item $P(n) :=  \Big\{\pmat{ a & b \\ c & -a^\top} \in \fsl(n+1|n+1,\C) \: b= b^\top,
c^\top = -c\Big\}$, $n > 1$. 
\item $G(3)$, $F(4)$, and $D(2,1,\alpha)$, where $\alpha\in \C\backslash\{0,-1\}$.
\end{itemize}

\begin{rem} \mlabel{rem:1.2}
The even parts of the complex classical {nonexceptional} Lie superalgebras are
\[ \begin{cases}
A(n,m)_\oe &\cong\ \ \fsl(n+1,\C) \oplus \fsl(m+1,\C) \oplus \C \\
A(n,n)_\oe &=\ \  \fsl(n+1,\C)^2
  \end{cases} \]
\and
\[ \osp(n|2m)_\oe \cong \so(n,\C) \oplus \sp(2m,\C), \qquad
Q(n)_\oe \cong \fsl(n+1,\C), \qquad P(n)_\oe \cong \fsl(n+1,\C).\]
In particular, the even part has non-zero (one-dimensional) center for
$A(n,m), n \not=m$ and $C(n)$.

(b) According to \cite{Pa80}, any real form $\g$ of a complex classical Lie superalgebra
$\g_\C$ is determined up to {isomorphism} by its even part $\g_\oe$.

(c) In view of \cite[Thm.~1, \S 2.3.4]{Ka77}, the Cartan--Killing form of
$A(n,n)$ and $Q(n)$ vanishes, whereas the Cartan--Killing form of
$A(n,m)$, $n \not=m$, and of $C(n)$ is non-degenerate. This further implies that
all derivations of $A(n,m)$ and $C(n)$ are inner (\cite[Prop.~2.3.4]{Ka77}). Since
derivations and cocycles are in one-to-one correspondence via the map
\[ \der(\g) \to Z^2(\g,\C), \quad D \mapsto \omega_D, \quad
 \omega_D(x,y) := \kappa(Dx,y) = \str(\ad(Dx) \ad y),\]
where $\kappa$ is the Cartan--Killing form, and
inner derivations correspond to trivial cocycles,
it follows that all central extensions of $A(n,m)$, $n \not=m$, and
$C(n)$ are trivial.
\end{rem}

\begin{exs} \mlabel{ex:1.3} (a) The prototype of a compact Lie superalgebra
is the real form
\[ \g := \fu(p|q;\C) := \Big \{
\pmat{A & B \\ iB^* & D} \in \gl(p|q;\C) \:
A \in \fu(p;\C), D \in \fu(q;\C), B \in M_{p,q}(\C) \Big\}\]
of the complex Lie {super}algebra $\gl(p|q;\C)$.
Then $\g_\oo \cong M_{p,q}(\C)$ and the $\g_\oe$-module structure
on this space is equivalent to $\C^p \otimes \oline{\C^q}$, so that
$\g_\oo$ is a simple $\g_\oe$-module.

An odd element $X \in \gl(p|q;\C)$ is contained in
$\fu(p|q;\C)$ if and only if $X^* = -iX$. Therefore the involution
on $\gl(p|q;\C)$ that leads to the real form
$\fu(p|q;\C)$ can be written as
\[ \sigma(X_0 + X_1) = - X_0^* + i X_1^* \quad \mbox{ for } \quad
X_0 \in \g_{\C,\oe}, X_1 \in \g_{\C,\oo}.\]
For later use we record the following property of $\g$:
\begin{equation}\label{eq:square}
 [X,X] = 2 X^2 = 2i X^*X \not=0  \quad \mbox{ for } \quad 0 \not= X \in \g_\oo.
\end{equation}

(b)  The Lie superalgebra $\g = \su(n|n;\C) := \{ X \in \fu(n|n;\C) \: \str X = 0\}$
has the non-zero center
$\fz = \R i\1 = \fz(\g_\oe)$ and $\g_{\bar 0} \cong \fz \oplus \su(n,\C)^2$, so that
$\fz$ acts trivially on $\g_{\bar 1}$. Since $M_n(\C) \cong \C^n \otimes \oline{\C^n}$
is a simple module of $\g_\oe$,
we see that for $\g$ and also for the simple
Lie superalgebra $\psu(n|n;\C) = \g/\z(\g)$, the representation of $\g_\oe$ on
$\g_\oo$ is irreducible.

We also note that
\[ i\1 = X^2 = \shalf [X,X] \quad \mbox{ for } \quad X := \pmat{0 & \1 \\ i\1 & 0} \in \g_{\bar 1},\]
shows that $\fz = \fz(\g) \subeq [\g,\g]$, that is,
$\g$ is a non-trivial central extension of $\psu(n|n;\C) := \g/\fz(\g)$.
Since the {center of the even part} of this quotient is trivial, it is not unitary
(cf.~Lemma~\ref{lem:form} below);
all unitary representations $\pi$ of $\g$ with $\pi(i\1) = 0$
kill the element $X$, hence are trivial
because the complexification $\psl(n|n;\C)$ of $\psu(n|n;\C)$ is a simple complex Lie superalgebra
(see also~\cite{Ja94}).

For $n \not=m$, the center $\fz$ of $\su(n|m;\C)_\oe$ is non-zero and acts non-trivially
on $\su(n|m;\C)_\oo$.

(c) For $n \geq 0$, the Lie superalgebra
\[ \fq(n) := \tilde Q(n) \cap \su(n+1|n+1;\C)
= \Big\{\pmat{ a & b \\ b & a}  \: a^* = -a, b^* = -i b, \tr b = 0\Big\},\]
is a compact real form of $\tilde Q(n)$. Accordingly, $\pq(n) := \fq(n)/\R i \1$
is a compact real form of the Lie superalgebra $Q(n)$, which is simple for
$n \geq 1$. Note that
\[ \{ b \in M_{n+1}(\C) \: b^* = -i b\} = (1-i) \fu(n+1,\C) \]
shows that the module structure for $\fq(n)_\oe
\cong \fu(n+1,\C)$ on $\fq(n)_\oo$ is equivalent
to the adjoint representation on $\su(n+1,\C)$, hence irreducible.

We also note that $\fz(\fq(n)_\oe) = \{0\}$,
that {$\fq(n)=[\fq(n),\fq(n)]$} and that it {embeds} naturally as a hyperplane ideal into
\[ \hat\fq(n) :=  \Big \{
\pmat{a & b \\ b & a} \in \gl(n+1|n+1;\C) \:
a^* = -a, b^* = -i b \Big\}.\]
Therefore
$D := \ad \pmat{0 & \1 \\ i\1 & 0}$
defines an exterior odd derivation of $\fq(n)$ with $D^2 = 0$.

(d) Let $\fc(n)$ be the compact real form of $C(n) = \osp(2|2n-2)$, $n \geq 2$, with
\[ \fc(n)_\oe
\cong \so(2,\R) \oplus \sp(n-1) \cong \R \oplus \fu(n-1;\H) \]
and
\[ \fc(n)_\oo \cong \bH^{n-1} \quad \mbox{ with}  \quad
(\fc(n)_\oo)_\C  \cong M_2(\C)^{n-1}\cong \C^2 \otimes \C^{2(n-1)}.\]
As $\H^{n-1}$ is an irreducible $\sp(n-1)$-module, the action of $\fc(n)_\oe$
on $\fc(n)_\oo$ is irreducible.

We also note that $\fz(\fc(n)_\oe) \cong \so(2,\R) \not=\{0\}$.
The Cartan--Killing form of $\fc(n)$ is non-zero, hence non-degenerate, and this implies that
all derivations are inner and that all central extensions are trivial
(cf.\ \cite[Prop.~2.3.4]{Ka77}).
\end{exs}

\begin{thm} \mlabel{thm:1.4} {\rm(Compact classical Lie superalgebras)}
Every compact simple real Lie superalgebra $\g$ is isomorphic to one of the following:
\begin{itemize}
\item[\rm(a)] $\su(n|m;\C)$ for $n > m \geq 1$.
\item[\rm(b)] $\psu(n|n;\C)$ for $n \geq 2$.
\item[\rm(c)] $\pq(n) = {\fq(n)/i\R \1}$, $n \geq 2$.
\item[\rm(d)] $\fc(n)$, the compact real form of $C(n) = \osp(2|2n-2)$, $n \geq 2$.
\end{itemize}
For all these simple real Lie
superalgebras $\g$, the $\g_{\oe}$-module $\g_\oo$ is irreducible,
$\z(\g_\oe) =\{0\}$ for ${\psu(n|n;\C)}$ and {$\pq(n)$}, and
$\z(\g_\oe) \cong \R$ for $\su(n|m;\C), n > m$ and $\fc(n)$.
\end{thm}

\proof The compactness of $\g$ implies that it is classical.
If $\g_\C$ is not simple, then $\g$ is a complex simple Lie superalgebra, considered as  real
one. Then $\ad \g_\oe$ is a complex subalgebra of $\der(\g)$ and this can only
generate a relatively compact group of automorphisms if it is trivial, i.e.,
$\g_\oe \subeq \fz(\g)$, but this contradicts the simplicity of $\g$.

Therefore $\g_\C$ is simple and $\g$ is a compact real form of the classical
Lie superalgebra $\g_\C$. From the classification
of the real forms in \cite[Thm.~2.5]{Pa80} it follows that
compact real forms are unique whenever they exist.
We also see that the exceptional algebras and the algebras
$P(n)$ have no compact real forms because
the Lie algebras $\fsl(2,\R)$, $\so(p,q)$, $pq > 1$,
and $\su^*(n)$, $n > 0$, are not compact.
Further, the algebras $B(n,m)$, $m > 0$,
and $D(n,m)$ have no compact real form because
the Lie algebras $\so^*(2n)$, $n > 1$,\begin{footnote}
{In $\so^*(2n) = \so(2n,\C) \cap \fsl(n,\H)$ the subalgebra $\fu(n,\C)$ is maximally
compact and equality implies $n = 1$.}
\end{footnote}
are never compact.

For $A(n,m)$, $Q(n)$ and $C(n)$, the compact real forms are given by {(a), (b), (c) and
(d)} respectively. The simplicity of the $\g_\oe$-module $\g_\oo$ follows
from the discussion in Examples~\ref{ex:1.3}, which also implies
the remaining assertion on $\fz(\g_\oe)$.
\qed

\begin{rem}  In \cite{F13}, Fioresi constructs compact real forms
in terms of the root decomposition. Fixing a Cartan subalgebra
and a Chevalley basis  $\{H_j,X_\a\}$ (see \cite[Chapter 3]{FG12}),
she thus obtains for any classical Lie {superalgebra} satisfying
$[X_\a,X_{\a}]=0$ for all $\a\in\Delta$ a compact real form.
\end{rem}

\section{Unitary Lie superalgebras}
\mlabel{sec:2}

In this section we take a closer look at the class of unitary
Lie superalgebras and some of their specific properties.

\begin{defn}\label{1new} A {(finite dimensional)} {\it unitary representation} for a finite dimensional Lie
superalgebra $\fg$ is a pair $(\cH, \rho)$ where $\cH = \cH_\oe \oplus \cH_\oo$
is a {finite dimensional} complex super Hilbert space with the corresponding sesquilinear positive definite even hermitian form
$\la\cdot,\cdot\ra:\cH\times\cH\to \C$, and $\rho:\fg\to \End_\C(\cH)$ is a (real) Lie superalgebra
homomorphism satisfying
\begin{equation}\label{new1}
\la \rho(X)v,w\ra=\la v,-i^{|X|}\rho(X)w\ra,
\end{equation}
for $X\in\fg$ and $v,w\in\cH$.
If we write
\[ \fu(\cH_\oe | \cH_\oo) :=
\Big\{ \pmat{ a & b \\ ib^* &  d} \in \gl(\cH) \:
a \in \fu(\cH_\oe), d \in \fu(\cH_\oo), b \in \Hom(\cH_\oo, \cH_\oe)\Big\}, \]
then unitary representations correspond to homomorphism of Lie superalgebras
\break $\rho \: \g \to \fu(\cH_\oe | \cH_\oo)$.
\end{defn}

Since we can form direct sums of unitary representations, we
have the following trivial lemma:

\begin{lem}\label{n1} For a finite dimensional Lie superalgebra, the
following are equivalent:
\begin{itemize}
\item[\rm(i)] The finite dimensional unitary representations
of $\g$ separate the points.
\item[\rm(ii)] $\g$ has a faithful finite dimensional unitary representation.
\item[\rm(iii)] $\g$ is isomorphic to a subalgebra of some
$\fu(p|q;\C)$.
\end{itemize}
\end{lem}

\begin{defn}\label{unitary} (a) We call a finite dimensional real Lie superalgebra
$\g$ {\it unitary} if it has a faithful finite dimensional unitary representation.

(b)  For a finite dimensional Lie superalgebra $\g$, we write
$\cC(\g)  \subeq \g_{\bar 0}$ for the closed convex cone generated by the elements of the form
$[X,X]$, $X \in \g_{\bar 1}$.
\end{defn}

\begin{lem} \mlabel{lem:form}
If $\g$ is unitary then the following assertions hold:
\begin{itemize}
\item[\rm(i)] $\g$ is compact, i.e., the even part $\g_\oe$ is compactly embedded in $\g$.
\item[\rm(ii)] $[X,X] \not=0$ for $0 \not=X \in \g_\oo$.
\item[\rm(iii)] $\cC(\g)$ is pointed, i.e., $\cC(\g) \cap - \cC(\g) = \{0\}$.
\item[\rm(iv)] There exists a $\g_0$-invariant linear functional $\omega \in \g_\oe^*$,
for which the form
\[ \kappa_\omega \: \g_\oo \times \g_\oo \to \R, \quad
\kappa_\omega(X_1, Y_1) := \omega([X_1,Y_1]) \]
is positive definite.
\item[\rm(v)] $\fz(\g_\oe) \not=\{0\}$.
\end{itemize}
\end{lem}

\proof (i)-(iii) We may assume that $\g \subeq \fh := \fu(p|q;\C)$.
Then (i) follows from the fact that $\fh_\oe$ is compactly embedded in $\fh$,
and (ii) from \eqref{eq:square}. Further \eqref{eq:square} implies that
\[ \cC(\g) \subeq \cC(\fh) \subeq \{ X \in \fu(p,\C) \oplus \fu(q,\C) \:
\Spec(-i X) \geq 0\},\]
which implies that $\cC(\fg)$ is pointed.

(iv) Since the cone $\cC(\g)$ is pointed, its dual cone
$\cC(\g)^\star \subeq \g_\oe^*$ has interior points.
Since the group $K$ generated by $e^{\ad \g_\oe}$ has compact closure by (i),
$\Int(\cC(\g)^\star)$ contains a fixed point $\omega$ for this group
(cf.\ \cite[Lemma 5.1.2]{NS11}).
Then $\omega(Y) > 0$ for any non-zero element $Y\in \cC(\g)$, and therefore
(ii), (iii) imply that $\kappa_\omega$ is positive definite on $\g_\oo$.
Further, the invariance of $\omega$ under $K$, and hence under $\g_\oe$,
implies that $\kappa_\omega$ is $\g_\oe$-invariant.

(v) Since $\omega$ is $\g_{\oe}$-invariant, it vanishes on $[\g_\oe,\g_\oe]$,
so that $\g_\oe = \fz(\g_\oe) \oplus [\g_\oe, \g_\oe]$ implies that the center is
non-zero.
\qed

Combining Lemma~\ref{lem:form}(v) with Theorem~\ref{thm:1.4}, we obtain immediately:

\begin{prop} \mlabel{prop:simp-nonunit}
The simple compact Lie superalgebras
$\psu(n|n;\C)$, $n > 0$, and $\pq(n)$, $n \geq 2$, are not unitary.
\end{prop}

\begin{rem}\label{new3}
(a) More examples of compact Lie superalgebras which are not unitary
can be obtained as follows.
A Lie superalgebra $\g$ is called a {\it Clifford--Heisenberg Lie superalgebra}
if $\g_\oe \subeq \z(\g)$. Then the bracket is completely determined by the
symmetric bilinear form
\[ \beta\: \g_{\bar 1} \times\g_{\bar 1} \to \g_\oe, \quad
(X,Y) \mapsto [X,Y]. \]
If $\dim \g_\oe = 1$ and $\beta$ is indefinite, then $\cC(\g) = \g_\oe$ is
not pointed, so that $\g$ is not unitary.

In view of Lemma~\ref{lem:form}, unitarity even requires that $\beta$
is positive or negative definite.
The spin representation associated to a
positive definite form $\beta$ implies that
$\g$ is unitary if $\beta$ is positive definite
(cf.\ Subsection~\ref{subsec:2.1})

(b) If $\g$ is a compact Lie superalgebra,
then $\g_{\bar 0}$ is in particular
a compact Lie algebra (\cite[Prop.~12.1.4]{HN12}). Hence the ideal
$[\g_{\bar 1},\g_{\bar 1}]$ of $\g_{\bar 0}$ has a complement
$\fc$, and
\[ \g \cong ([\g_{\bar 1},\g_{\bar 1}] \oplus \g_{\bar 1}) \rtimes \fc \quad \mbox{ with } \quad
[\fc, {[\g_\oo, \g_\oo]}] = \{0\}.\]

If, conversely, $\fc$ is a compact Lie algebra and $\fh$ a compact
Lie superalgebra on which $\fc$ acts by even derivations
$\alpha \: \fc \to \der(\fh)_\oe$ with
$\alpha(\fc)(\fh_\oe) = \{0\}$ and
$e^{\alpha(\fc)_\oo} \subeq \GL(\fh_\oo)$ relatively compact, then
$\fh \rtimes \fc$ is a compact Lie superalgebra.
Therefore, for most issues concerning the structure theory of compact Lie superalgebras,
it suffices to assume that $\g$ is generated by $\g_\oo$.
\end{rem}


\subsection{The spin representation}
\mlabel{subsec:2.1}

Let $V$ and $\cH$ be finite dimensional complex Hilbert spaces.
A representation of the {\it canonical anticommutation relations (CAR)}
is an antilinear map $a \: V \to B(\cH)$ satisfying
\[  a(f)a(g) + a(g) a(f) = 0 \quad \mbox{ and } \quad
a(f)a(g)^* + a(g)^* a(f) = \la g,f \ra \1 \quad \mbox{ for } \quad f,g \in V.\]

For every complex Hilbert space $V$, we obtain a representation of
the CAR on the exterior algebra $\Lambda(V) = \bigoplus_{k = 0}^\infty \Lambda^k(V)$,
endowed with the natural scalar product, by
\[ a_0(f)^*(f_1 \wedge \cdots \wedge f_n) = f \wedge f_1 \wedge \cdots \wedge f_n\]
and
\[ a_0(f)(f_1 \wedge \cdots \wedge f_n) =\sum_{j = 1}^n (-1)^{j+1}
\la f_j, f \ra f_1 \wedge \cdots\wedge \hat{f_j} \wedge\cdots \wedge f_n\]
(cf.\ \cite{Ot95}).

The complex {\it Clifford--Heisenberg Lie superalgebra} associated to the complex
Hilbert space $V$ is the Lie superalgebra
\[ \cheis(V) := \C \oplus V \oplus \oline V, \quad
[(z,v,w), (z',v',w')] := (\la v, w'\ra + \la v',w\ra,0,0).\]
Any representation of the CAR leads to a complex linear representation of
$\cheis(V)$ by
\[ \rho(z,v,w) := z \1 + a(v)^* + a(w).\]
Accordingly, we obtain a real form of $\cheis(V)$ by
\[ \fh := \{ (zi,v,-iv) \: z\in \R, v \in V \}\]
and any representation of the CAR defines a unitary representation of $\fh$.
For $(0,v,-iv) \in \fh$, we have
\[ [(0,v{,-iv}), (0,v{,-iv})]
= (\la v, -iv \ra + \la v,-iv \ra, 0,0)
= (2i\|v\|^2, 0,0).\]

We now consider the Lie superalgebra
\[ \g := \fh \rtimes_D \R, \quad D(iz,v,-iv) := (0,iv, v), \quad z \in \R, v \in V.\]
Then $d := (0,1) \in \fz(\g_{\oline 0})$ acts non-trivially on $\g_{\oline 1}$.
The Fock representation of the CAR extends to $\g$ by
\[ \rho(d)(f_1 \wedge \cdots \wedge f_n) := in f_1 \wedge \cdots \wedge f_n.\]
We thus obtain for $\dim V = n$ a unitary irreducible representation
of $\g$ satisfying $\Spec(-i\rho(d)) = \{0,\ldots, n\}$. {For a detailed study of spin representations see
\cite{Va04}.}

\subsection{Central extensions}

In this subsection we discuss some aspects of central
extensions of real Lie superalgebras that will be needed in
Section~\ref{sec:4} below.

\begin{lem} If $\beta \: \g_{\oo} \times \g_{\oo} \to \R$ is a
$\g_\oe$-invariant symmetric bilinear form, i.e.,
\[ \beta([X,Y],Z) + \beta(Y,[X,Z]) = 0 \quad \mbox{ for} \quad
X \in \g_\oe, Y,Z \in \g_\oo, \]
then
\begin{equation}
  \label{eq:cocyc}
 \omega(X_0 + X_1, Y_0 + Y_1) := \beta(X_1, Y_1)
\end{equation}
is a Lie superalgebra  $2$-cocycle, i.e.,
\[ \hat\g := \R \oplus \g \quad \hbox{ with }
\quad [(z,X), (z', X')] := (\omega(X,X'), [X,X']) \]
is a Lie superalgebra with $\hat\g_\oe = \R \oplus \g_\oe$ and
$\hat\g_\oo = \g_\oo$.
\end{lem}

\begin{proof}
 We have to check for homogeneous elements $X,Y,Z \in \g$ the relation
\[ \omega(X,[Y,Z]) = \omega([X,Y], Z) + (-1)^{|X| \cdot |Y|} \omega(Y,[X,Z]). \]
Here the only non-trivial case is the situation where one argument, say
$X$, is even and the others are odd. Then
\[ \omega(X,[Y,Z]) = 0 = \beta([X,Y], Z) + \beta(Y,[X,Z])
= \omega([X,Y],Z) + \omega(Y,[X,Z]) \]
follows from the invariance of $\beta$ under $\g_\oe$.\end{proof}

\begin{prop} \mlabel{prop:2.9} If $\beta$ is a positive definite
$\g_\oe$-invariant form on $\g_\oo$
and $\hat\g$ is the central extension defined by the cocycle
\eqref{eq:cocyc}, then the cone $\cC(\hat\g)$ is pointed.
\end{prop}

\begin{proof} The convex cone $\cC(\hat\g)$ is generated by the elements
$(\beta(X,X), [X,X])$, $X \in\g_\oo$, and if $\beta$ is positive definite,
it suffices to consider elements with $\beta(X,X) = 1$. This leads to the compact
subset
\[ C  := \{ (1, [X,X]) \: X \in \g_\oo, \beta(X,X) = 1\} \subeq\R \times \g_\oe\]
whose convex hull does not contain $0$. Therefore it generates a pointed convex cone.
\end{proof}

\begin{cor} Every compact Lie superalgebra $\g$ has a compact central extension
$\hat\g$ for which the cone $\cC(\hat\g)$ is pointed.
\end{cor}

\proof Since $e^{\ad\g_\oe}$ has compact closure, $\g_\oo$ carries a
positive definite $\g_\oe$-invariant
symmetric bilinear form $\beta$. Then the corresponding
centrally extended Lie superalgebra $\hat\g$ has a pointed cone $\cC(\hat\g)$
by Proposition~\ref{prop:2.9}.
\qed

\section{Decomposition theory}
\mlabel{sec:3}

In this section we turn to the structure of a
reductive Lie superalgebra $\g$ over a field $\K$ of characteristic zero.
{Since the results in this section will be applied to unitary (and so compact) Lie superalgebras, and in this case
$\g$ is always} a semidirect sum of an ideal of $\g_\oe$ and the ideal
of $\g$ generated by  $\g_\oo$ (cf.\ Remark~\ref{new3}(b)),
we assume in the following that
\begin{equation}
  \label{eq:g1gen}
\g_\oe = [\g_\oo, \g_\oo].
\end{equation}
The main goal of this section is to decompose $\g$
into center and ideals which are central extensions of simple Lie superalgebras
or of algebras of the form $\fk \otimes \Lambda_1$, where $\fk$ is a
simple Lie superalgebra.

Let
\[ \fb:= \fz_{\g_\oo}(\g_\oe) = \{Y\in\g_{\bar1}\mid[Y,\g_{\bar0}]=\{0\}\}.\]
We note that $[[\fb,\fb],\g_{\bar 1}]\subseteq [\fb,\g_{\bar 0}]=\{0\}$, and since
$\g$ is generated by $\g_\oo$, this leads to
\begin{equation}\label{miss2}
[\fb,\fb]\subseteq \fz(\g).
\end{equation}
Since $\fb$ is $\g_{\bar0}$-invariant and $\g_\oo$ is a semisimple $\g_\oe$-module,
we have a $\g_\oe$-invariant decomposition
$\g_{\bar1}=\fb\oplus\fa$ with $\fa=\bigoplus_{j\in J}\fa_j$, where
the $\fa_j$ are  simple $\g_{\bar0}$-submodules of $\g_\oo$. We
then have
\begin{equation}\label{m12-4}
[\aij,\fa_k]\subseteq\fa_k\cap(\fa_i+\fa_j)=\{0\}\quad\mbox{ for }\quad
k\not\in\{i,j\},
\end{equation}
and thus in particular
\begin{equation}\label{m12-4b}
[[\fa_i, \fa_i],\fa_k] = \{0\} \quad\mbox{ for }\quad
k\not= i.
\end{equation}
We likewise obtain
\begin{equation}\label{m12-4c}
[[\fb, \fa_j], \fa_k] \subeq \fa_k \cap (\fb + \fa_j) = \{0\} \quad\mbox{ for }\quad
k\not = j.
\end{equation}

Note that
\begin{equation}
  \label{eq:g0sum}
  \g_\oe = [\g_\oo, \g_\oo]
= [\fb,\fb] + [\fb,\fa] + [\fa, \fa]=
\underbrace{[\fb,\fb]}_{\subeq \fz(\g)} + \sum_j [\fb,\fa_j] + \sum_{j,k} [\fa_j, \fa_k].
\end{equation}
Since the  spaces $[[\fb,\fa_j],\fa_i]$ and $[[\fa_j, \fa_k], \fa_i]$ are $\g_\oe$-submodules of
$\fa_i$, they are either zero or equal to $\fa_i$. As $[\fb,\fb]$ is central by \eqref{miss2},
the relations $\fa_i=[\g_{\bar0},\fa_i]$, \eqref{m12-4},
\eqref{m12-4c} and  \eqref{eq:g0sum} imply  that
\begin{equation}
  \label{eq:m12-4d}
[[\fb, \fa_i],\fa_i] = \fa_i \quad \mbox{ or } \quad
[[\fa_j, \fa_i],\fa_i] = \fa_i \quad \mbox{ for some } j \in J.
\end{equation}

We put
\[ J_s:=\{k\in J\: [[\fa,\fa],\fa_k]=\fa_k\}
=\{k\in J\: (\exists k' \in J)\ [[\fa_k,\fa_{k'}],\fa_k]=\fa_k\}.\]
Let $J_a:=J\setminus J_s$ and consider  $k\in J_a$. Then $[[\fa,\fa],\fa_k]=\{0\}$
by definition and so $[[\fb,\fa],\fa_k]=[[\fb,\fa_k],\fa_k]=\fa_k$ follows from
\eqref{m12-4c}. Therefore \eqref{eq:m12-4d} leads to
\begin{equation}\label{teh2}
J_a=\{k\in J\: [[\fb,\fa_k],\fa_k]=\fa_k\hbox{ and }\
(\forall j\in J)\ [[\fa_k,\fa_j],\fa_k]=\{0\}\}.
\end{equation}

\begin{lem}\label{teh5}
For each $k\in J_s$, there exists a unique $k'\in J$ such that
$\fa_k=[[\fa_k,\fa_{k'}],\fa_k]$. Then $k'\in J_s$ and
$\fa_{k'}=[[\fa_{k'},\fa_k],\fa_{k'}]$.
\end{lem}

\begin{proof} Let $k\in J_s$ and $k'\in J$ such that
$\fa_k=[[\fa_k,\fa_{k'}],\fa_k]$.

\nin{\bf Step 1:} First we show that $k' \in J_s$.
If $k' \in J_a$, then $k \not=k'$, so that
\[ [\fa_k,\fa_{k'}] = [\fa_k,[[\fb,\fa_{k'}],\fa_{k'}]]
\subeq [[\fa_k,\fa_{k'}],[\fb,\fa_{k'}]]+
[\underbrace{[\fa_k,[\fb,\fa_{k'}]]}_{=0\ \text{by}\ \eqref{m12-4c}},\fa_{k'}]\subeq [\fb,\fa_{k'}].\]
With \eqref{m12-4c}, this leads to the contradiction
$\fa_k =[[\fa_k,\fa_{k'}],\fa_k]
\subeq [[\fb,\fa_{k'}], \fa_k] = \{0\}.$

\nin {\bf Step 2:} Next we show that
$\fa_{k'}=[\fa_{k'},[\fa_k,\fa_{k'}]]$. We may assume that $k\not=k'$.  As $k'\in J_s$ by
Step $1$, there exists a $t\in J$ such that $\fa_{k'}=[[\fa_{k'},\fa_t],\fa_{k'}]$.
If $t=k'$, then
\[ [\fa_k,\fa_{k'}] = [\fa_k,[[\fa_{k'},\fa_{k'}],\fa_{k'}]]
\subeq [\underbrace{[\fa_k,[\fa_{k'},\fa_{k'}]]}_{=0},\fa_{k'}]+
[[\fa_k,\fa_{k'}],[\fa_{k'},\fa_{k'}]]
\subeq [\fa_{k'},\fa_{k'}] \]
leads with \eqref{m12-4b} to the contradiction
\[ \fa_k=[[\fa_k,\fa_{k'}],\fa_k]
\subeq [[\fa_{k'},\fa_{k'}],\fa_k] = \{0\}.\]
Thus $t\not=k'$.
If $t\not= k$, then $t,k,k'$ are distinct and  so
\[ [\fa_k,\fa_{k'}] = [\fa_k,[[\fa_{k'},\fa_t],\fa_{k'}]]
\subeq [\underbrace{[\fa_k,[\fa_{k'},\fa_{t}]]}_{=0},\fa_{k'}]
+ [[\fa_k,\fa_{k'}],[\fa_{k'},\fa_{t}]]
\subeq [\fa_{k'},\fa_{t}] \]
leads to the contradiction
$\fa_k = [[\fa_k,\fa_{k'}],\fa_k] \subeq
[[\fa_t,\fa_{k'}],\fa_k] = \{0\}$.
We conclude that $t = k$.

\nin{\bf Step 3:} Now  we show the uniqueness of $k'$. Suppose that $s\not=t\in J$ satisfy
\[ \fa_k=[\fa_k,[\fa_k,\fa_{s}]]=[\fa_k,[\fa_k,\fa_t]].\]
By the symmetry proved in the previous step,
\[ \fa_s=[\fa_s,[\fa_s,\fa_k]] \quad \mbox{ and } \quad \fa_t=[\fa_t,[\fa_t,\fa_k]].\]
This leads to
\begin{align*}
[\fa_k,\fa_s]&=[[\fa_k,[\fa_t,\fa_k]],\fa_s]
\subseteq [[\fa_k,\fa_s],[\fa_t,\fa_k]]+[[[\fa_t,\fa_k],\fa_s],\fa_k]\\
&\subseteq [\fa_t,\fa_k]+[[[\fa_s,\fa_t],\fa_k],\fa_k]+[[[\fa_s,\fa_k],\fa_t],\fa_k]\\&
\subseteq [\fa_t,\fa_k]+[\fa_k,\fa_k].
\end{align*}
Similarly,
\begin{equation}
  \label{eq:10}
[\fa_k,\fa_t]\subseteq [\fa_s,\fa_k]+[\fa_k,\fa_k].
\end{equation}
If $k\not\in\{s,t\}$, then (\ref{m12-4}) and \eqref{m12-4b}
lead to the contradiction
\[ \fa_s=[\fa_s,[\fa_s,\fa_k]]\subseteq [\fa_s,[\fa_t,\fa_k]] + [\fa_s,[\fa_k,\fa_k]]=\{0\}.\]
If $k=s$, then $s\not=t$ and \eqref{eq:10} yield
\[ \fa_t=[\fa_t,[\fa_t,\fa_k]]\subeq  [\fa_t,[\fa_s,\fa_k]] + [\fa_t,[\fa_k,\fa_k]]=\{0\},\]
which is absurd.
By symmetry, a similar argument works for the case $k=t$. This shows that
$t = s$ and the proof is complete.
\end{proof}

\begin{lem}\label{teh4}
$[[\fa_k,\fa_k],\fa_k]=\{0\}$ if and only if $[\fa_k,\fa_k]\subseteq\fz(\g)$.
This holds in particular for $k\in J_a$ and  $k \in J_s$ with $k \not=k'$.
\end{lem}

\begin{proof} Clearly $[\fa_k,\fa_k]\subseteq\fz(\g)$ implies
$[[\fa_k,\fa_k],\fa_k]=\{0\}$. If, conversely,
$[[\fa_k,\fa_k],\fa_k]=\{0\}$, then \eqref{m12-4b} leads to
\[ [[\fa_k,\fa_k],\g_{\bar1}]=[[\fa_k,\fa_k],\fa]=[[\fa_k,\fa_k],\fa_k]=\{0\},\]
so that  $[\fa_k,\fa_k]\subseteq\fz(\g)$ follows from
$\g_{\bar0}=[\g_{\bar1},\g_{\bar1}]$.

If $k \in J_a$, then $[[\fa_k,\fa_k],\fa_k]=\{0\}$ by definition.
If $k \in J_s$ and $k\not= k'$, then the uniqueness of $k'$
implies that $[[\fa_k, \fa_k], \fa_k] = \{0\}$.
\end{proof}

\begin{lem}\label{teh6}
\begin{itemize}
  \item[\rm(i)]  If $k\in J_s$, then $[\fb,\fa_k]\subseteq [\fa_k,\fa_{k'}] \cap [\g_\oe,\g_\oe]$.
  \item[\rm(ii)] $[\fa_k,\fa_s]=\{0\}$ if $k\in J_a$ and $s\not=k$,
or $k\in J_s$ and $s\not\in\{k,k'\}$.
  \item[\rm(iii)] If $j,k \in J_s$ and $k \not\in \{j,j'\}$, then
$[\fa_k,\fa_{k'}] \cap [\fa_j, \fa_{j'}] \subeq \fz(\g)$.
\end{itemize}
\end{lem}

\begin{proof} (i) follows from
\[ [\fb,\fa_k]=[\fb,[\fa_k,[\fa_k,\fa_{k'}]]]\subseteq
[[\fb,\fa_k],[\fa_k,\fa_{k'}]]+
[\fa_k,\underbrace{[\fb,[\fa_k,\fa_{k'}]]}_{=0}]\subseteq [\fa_k,\fa_{k'}] \cap [\g_\oe,\g_\oe].\]

(ii) Assume first that $k\in J_s$ and $s\not\in\{k,k'\}$.
Then
\begin{align*}
[\fa_s,\fa_k]&=[\fa_s,[\fa_k,[\fa_k,\fa_{k'}]]]
\subseteq [[\fa_s,\fa_k],[\fa_k,\fa_{k'}]]+[\underbrace{[\fa_s,[\fa_k,\fa_{k'}]]}_{=0},\fa_k]\\
&\subseteq [\fa_k,\fa_{k'}]\cap [\fg_{\bar0},\g_{\bar0}].
\end{align*}
If $s\in J_s$, then $s\not\in\{k,k'\}$ implies $k\not\in \{s,s'\}$ and
we get similarly
\[ [\fa_s,\fa_k]\subseteq [\fa_s,\fa_{s'}] \cap [\g_\oe,\g_\oe].\]
Thus $[\fa_s,\fa_k]\subseteq [\fa_s,\fa_{s'}]\cap [\fa_k,\fa_{k'}]\cap[\g_{\bar0},\g_{\bar0}]$, and
(\ref{m12-4b}) thus leads to
$[[\fa_s,\fa_k],\g_\oo] = \{0\}$, which in turn shows that
\[ [\fa_s,\fa_k] \subeq [\g_\oe, \g_\oe] \cap \fz(\g)
\subseteq [\g_{\bar0},\g_{\bar0}]\cap\fz(\g_{\bar0})=\{0\}.\]
Here we use that $\g_{\bar0}$ is reductive.

If $s\in J_a$, then
\begin{align*}
[\fa_k,\fa_s]&=[\fa_k,[\fa_s,[\fa_s,\fb]]]
\subseteq[[\fa_k,\fa_s],[\fa_s,\fb]]+\underbrace{[[\fa_k,[\fa_s,\fb]]}_{=0\ \text{by} \eqref{m12-4c}},\fa_s]\\
&= [[\fa_k,\fa_s],[\fa_s,\fb]] \subseteq
[\underbrace{[[\fa_k,\fa_s],\fa_s]}_{=0\ \text{by } \eqref{teh2}},\fb]
+[\underbrace{[[\fa_k,\fa_s],\fb]}_{=0},\fa_s] = \{0\}.
\end{align*}
The case $k \in J_a$ and $s \in J_s$ follows from the first part.

Finally assume that $k,s\in J_a$, $k\not=s$. Then
\[ [\fa_k,\fa_s]=[\fa_k,[\fa_s,[\fa_s,\fb]]]\subseteq
[[\fa_k,\fa_s],[\fa_s,\fb]] + [\fa_s, \underbrace{[\fa_k,[\fa_s, \fb]]}_{=0\ \text{by } \eqref{m12-4c}}]
\subseteq [\fa_s,\fb]\cap [\g_{\bar0},\g_{\bar0}].\]
We likewise get
$[\fa_k,\fa_s]\subseteq [\fa_k,\fb]$, so that \eqref{m12-4c} leads to
$[\fa_k,\fa_s] \subeq \fz(\g)$, and thus to
$[\fa_k,\fa_s]\subseteq  \fz(\g_{\bar0})\cap [\g_{\bar0},\g_{\bar0}]=\{0\}$.

(iii) Let $\fz := [\fa_k,\fa_{k'}] \cap [\fa_j, \fa_{j'}]$.
Then \eqref{m12-4b} shows that
\[ [\fz,\g_\oo] = [\fz, \fa_k +\fa_{k'}] \subeq [[\fa_i, \fa_{j'}], \fa_k+\fa_{k'}] = \{0\},\]
so that the assertion follows from $\g_\oe = [\g_\oo, \g_\oo]$.
\end{proof}

\begin{rem} \mlabel{rem:subalg} If $\fm \subeq \g_\oo$ is a $\g_\oe$-invariant subspace, then
$\fm + [\fm,\fm] \subeq \g$ is a $\g_\oe$-invariant subalgebra of $\g$.
In particular, $[\fm,\fm] \trile\g_\oe$ is an ideal.
\end{rem}

We now consider the following subalgebras of $\g$ (cf.\ Remark~\ref{rem:subalg}):
\[ \begin{array}{ll}
\g(k):=\fa_{k}+\fa_{k'}+[\fa_k, \fa_k] +[\fa_k, \fa_{k'}] + [\fa_{k'}, \fa_{k'}]\quad \mbox{ for }\quad  k\in J_s,\\
\fc(k):=\fa_k+[\fa_k,\fa_k] + [\fb,\fa_k] \qquad \mbox{ for } \quad k\in J_a.
\end{array}\]

For $k\not= k'$ in $J_s$, the relation $[\fa_k, \fa_k] + [\fa_{k'},\fa_{k'}] \subeq \fz(\g)$
(Lemma~\ref{teh6}) implies that
\[ \g(k) \cap [\g_\oe, \g_\oe] \subeq [\fa_k, \fa_{k'}].\]
From Lemma \ref{teh6}(i), we obtain
\[ [\fb,\fa]=\sum_{k\in J_a}[\fb,\fa_k]+\sum_{k\in J_s}[\fb,\fa_k]\subseteq \sum_{k\in J_a}[\fb,\fa_k]
+\sum_{k\in J_s}[\fa_k,\fa_{k'}]. \]
By Lemmas~\ref{teh4} and \ref{teh6}(ii) we have
\[ [\fa,\fa]
=\sum_{k\in J_a}[\fa_k,\fa_k]
+\sum_{k'\not=k\in J_s}[\fa_k,\fa_k]
+ \sum_{k\in J_s}[\fa_k,\fa_{k'}]
\subeq \fz(\g) + \sum_{k\in J_s}[\fa_k,\fa_{k'}].\]
Thus
\begin{align*}
\g&=\underbrace{[\fb,\fb]+[\fb,\fa]+[\fa,\fa]}_{\g_\oe}+\underbrace{\fb+\fa}_{\g_\oo}
=\fb + [\fb,\fb] +\sum_{k\in J_a}\fc(k)+\sum_{k\in J_s}\g(k).
\end{align*}

{The structure of subalgebras $\g(k)$ and $\fc(j)$, $k\in J_s, j\in J_a$ and possibilities for
the index sets $J_s$ and $J_a$ will be discussed in what follows, see for example Proposition \ref{prop:3.5}(iv), Example \ref{ex:tk}, Remark \ref{new-late} and Example \ref{ex:tildetk}.}

\begin{prop} \mlabel{prop:3.5}
For every $j \in J_s$, the Lie superalgebra $\g(j)$ is an ideal of $\g$
with the following properties:
\begin{itemize}
\item[\rm(i)] $\fa_j$ and $\fa_{j'}$ are simple $\g(j)_\oe$-modules.
In particular, $\g(j)_\oo$ is a semisimple $\g(j)_\oe$-module.
\item[\rm(ii)] $\z_{\g(j)_\oo}(\g(j)_\oe) = \{0\}$,
$\fz(\g(j)) \subeq \fz(\g)_\oe$ and $\fz(\g(j)_\oe) \subeq \fz(\g_\oe)$.
\item[\rm(iii)] $\pg(j) := \g(j)/\fz(\g(j))$ is a classical Lie superalgebra.
\item[\rm(iv)] If $\K = \R$ and $\g(j)$ is compact,
then $j = j'$. This holds in particular for every $j\in J_s$
if $\g$ is compact.
\end{itemize}
\end{prop}

\proof First we observe that
$[\fb, \fa_j + \fa_{j'}] \subeq [\fa_j, \fa_{j'}]$
by Lemma~\ref{teh6}(i). For $k \not\in \{j,j'\}$ we get
$[\fa_k, \fa_j + \fa_{j'}] =\{0\}$
from Lemma~\ref{teh6}(ii). This leads to
\[ [\g_\oo, \fa_j + \fa_{j'}] = [\fb + \fa_j + \fa_{j'}, \fa_j + \fa_{j'}] \subeq \g(j).\]
As $\g_\oe = [\g_\oo, \g_\oo]$ and $\g(j)$ is $\g_\oe$-invariant, this implies that
$\g(j) \trile \g$.

(i) Since all subspaces $\fa_k$, $k \not\in \{j,j'\}$, commute with
$\g(j)$ by Lemma~\ref{teh6}(ii) and all subspaces
$[\fb, \fa_k]$, $k \not\in \{j,j'\}$, commute with $\g(j)$ by \eqref{m12-4c},
Lemma~\ref{teh6}(i) leads to
\[ \g_\oe
\subeq \fz_{\g_\oe}(\g(j)) + \g(j)_\oe + [\fb, \fa_j + \fa_{j'}]
\subeq \fz_{\g_\oe}(\g(j)) + \g(j)_\oe. \]
Therefore $\fa_j$ and $\fa_{j'}$ are simple $\g(j)_\oe$-modules.

(ii) From the non-triviality of the simple $\g(j)_\oe$-modules
$\fa_j$ and $\fa_{j'}$, it follows that
the centralizer of $\g(j)_\oe$ in $\g(j)_\oo$ vanishes.
In particular, $\fz := \fz(\g(j)) \subeq \g(j)_\oe$.

From Lemma~\ref{teh6} it follows that
$[\fz,\g_\oo] = [\fz,\fa] = [\fz,\fa_j + \fa_{j'}]= \{0\}$,
so that $\fz \subeq \fz(\g)$ because $\g_\oo$ generates $\g$.

As $\fz(\g(j)_\oe) \trile \g(j)_\oe$ is $\g_\oe$-invariant, it is an abelian
ideal of the reductive Lie algebra $\g_\oe$, hence central in $\g_\oe$.

(iii)  Let $\fh := \g(j)$ and $\{0\}\not= I \trile \fh$ be an ideal.
If $I_{\bar1}=\{0\}$, then $I=I_{\bar0}\subseteq\fz(\h)$.

Suppose that $I_\oo \not=\{0\}$.
If $I_{\bar0}\subseteq\fz(\h)$, then $[I_{\bar1},\h_{\bar1}]\subseteq I_{\bar 0}
\subseteq\fz(\h)$, which leads to
$[I_{\bar1},\h_{\bar0}]=[I_{\bar1},[\h_{\bar1},\h_{\bar1}]]=\{0\}$.
In view of (ii), this contradicts $I_\oo \not=\{0\}$.
We conclude that $I_{\bar 0}\not\subseteq\fz(\h)$.
This implies that either $[I_{\bar0},\fa_j]\not=\{0\}$ or $[I_{\bar0},\fa_{j'}]\not=\{0\}$.
By symmetry, we may assume $[I_{\bar0},\fa_j]\not=\{0\}$.
Then $\{0\}\not=[I_{\bar0},\fa_{j}]\subseteq I_{\bar1}\cap\fa_{j}$.
Since the latter is a nonzero $\g(j)_\oe$-submodule of
the simple $\g(j)_\oe$-module $\fa_j$, we obtain $\fa_j\subseteq I_{\bar1}$.
We now get
\[ \fa_{j'}=[\fa_{j'},[\fa_{j'},\fa_j]]\subseteq [\fa_{j'},[\fa_{j'},I_{\bar 1}]]
\subseteq I_{\bar1},\]
and thus $I=\fh$. This proves (iii).

(iv) Since $\g(j)$ is compact, the same holds for the quotient
Lie superalgebra $\pg(j)$ which is classical by (iii). Hence the assertion
follows from Theorem~\ref{thm:1.4}.
\qed

\begin{ex} \mlabel{ex:tk} (Tangent superalgebras)
Let $\fk$ be a simple Lie algebra. Then the tangent algebra
$T\fk \cong  \fk \otimes \Lambda_1$, $\Lambda_1 = \K[\xi]$, $\xi^2 = 0$,
carries the structure of a
Lie superalgebra, where $(T\fk)_\oe = \fk \otimes 1$ and $(T\fk)_\oo = \fk \otimes \xi$.
Note that $\fk \otimes \xi$ is an abelian ideal.
Here the brackets of odd elements vanish.
Since $T\fk \cong \fk^2$ as a module of $\fk \cong (T\fk)_\oe$, this Lie algebra
is reductive. It has an odd derivation
\[ D = \frac{\partial}{\partial \xi} \: T\fk \to T\fk, \quad D(x \otimes 1 + y \otimes \xi)
= y \otimes 1.\]
Now
\[ \hat T\fk := T\fk \rtimes_D \R \]
also is a reductive Lie superalgebra and $(\hat T\fk)_\oo \cong \fk \oplus \R$ contains a
subspace commuting with $\fk$ which is not central.

In the terminology from above, we then have $\fb = \R (0,1)$,
$\fa = \fk \otimes \xi$, and $J = J_a$ is a one-element set because $[\fa,\fa] = \{0\}$.
\end{ex}

\begin{prop} Let $k \in J_a$.
Then $\fc(k) \trile \g$ is an ideal with the following properties:
\begin{itemize}
\item[\rm(i)] $\fk_k := [\fb,\fa_k]$ is a simple Lie algebra.
\item[\rm(ii)] $\fc(k)/\fz(\fc(k)) \cong T \fk_k$.
\item[\rm(iii)] $\fc(k)$ contains the nilpotent ideal $\fa_k + [\fa_k,\fa_k]$,
hence is not semisimple.
\item[\rm(iv)] Let $b \in \fb$ with $[b,\fa_k] \not=\{0\}$.
Then $\hat\fc(k) := \fc(k) + \K b+ \K[b,b]$ is a subalgebra of $\g$
for which $\hat\fc(k)/\fz(\hat\fc(k)) \cong \hat T\fk_k$ is semisimple.
\end{itemize}
\end{prop}

\proof (i), (ii)
Lemma~\ref{teh6}(ii) implies that $[\fa_k,\fa_j] = \{0\}$ for $j \not=k$.
By \eqref{miss2} and \eqref{m12-4c},
the centralizer of $\fa_k$ in $\g_\oe$ contains
\[ [\fb,\fb], \quad [\fa,\fa], \quad [\fb, \fa_j], j \not=k.\]
We conclude that $\g_\oe = \fz_{\g_\oe}(\fa_k) + [\fb, \fa_k]$, so that
$\fa_k$ is also a simple module of the ideal $\fk := [\fb,\fa_k]$ of the reductive
Lie algebra $\g_\oe$ and that $\fk$ has an ideal complement acting trivially on $\fa_k$.

For every $b \in \fb$, the operator $\ad b \: \fa_k \to \fk$ is $\fk$-equivariant.
Therefore the adjoint representation of $\fk$ is a sum of simple modules
isomorphic to $\fa_k$, hence semisimple and a direct sum of submodules isomorphic
to $\fa_k$. In particular, $\fk$ is semisimple, hence a direct sum of
simple ideals $\fk_1, \ldots, \fk_n$. Then the ideals $\fk_j$ are
simple $\fk$-modules with respect to the adjoint action and for
$i \not=j$ they are non-isomorphic because $\fk_i$ acts trivially on $\fk_j$
and vice versa. This implies that $n = 1$ and that $\fk$ is a simple
Lie algebra. Using Lemma~\ref{teh4}(ii), we see that
$\fc(k)$ is a central extension of the Lie superalgebra
\[ \fc(k)/\fz(\fc(k)) \cong  \fk + \fa_k \]
which is isomorphic to the Lie superalgebra $T\fk$ from Example~\ref{ex:tk}.

Note that $\fk$ annihilates a  complement of $\fa_k$ in $\g_\oo$,
$[\fa_k,\fa_k]$ is central, and $[\fa_k, \g_\oo] = [\fa_k,\fa_k]+ [\fa_k,\fb]$
by Lemma~\ref{teh6}(ii). We conclude that
\[ [\fc(k),\g_\oo] = [\fa_k, \g_\oo] + [\fk, \g_\oo] \subeq \fc(k),\]
and therefore $\fc(k) \trile \g$ is an ideal.

(iii) follows from $[\fa_k,\fa_k] \subeq \fz(\g)$ (Lemma~\ref{teh4}).

(iv) Let $\fk := \fk_k$. Since $\fc(k)$ is an ideal of $\g$, the subspace
$\hat\fc(k)$ is a subalgebra. Now it follows from Lemma \ref{teh4}, (\ref{miss2}) and (i) that $\fz(\hat\fc(k)) = [\fa_k,\fa_k]+\K[b,b]$
and $\hat\fc(k)/\fz(\hat\fc(k)) \cong T\fk \rtimes_D \fk$ for a derivation
mapping $D \in \der(T\fk)$ with $D^2 =0$ mapping
$\fa_k$ in a $\fk$-equivariant fashion to $\fk$. This implies that
it is isomorphic to $\hat T\fk$. To see that $\hat T\fk$ is semisimple,
let $\fm \trile \hat T\fk$ be a solvable ideal. Then
$\fm_\oe$ must be an abelian ideal of $(\hat T\fk)_\oe \cong \fk$, hence
trivial. Therefore $\fm \subeq (\hat T\fk)_\oo$ and the ideal property
implies $[b,\fm] = \{0\}$, which leads to $\fm \subeq \K b$.
As $[b,\fa_k] \not=\{0\}$, we obtain $\fm = \{0\}$.
\qed

In any Lie superalgebra $\g$ we have
$[\g,\g]_\oo = [\g_\oe, \g_\oo]$. If, in addition, $\g$ is reductive, then
the $\g_\oe$-module $\g_\oo$ is semisimple, so that
$\fz(\g)_\oo$ intersects $[\g,\g]$ trivially.
Therefore $\fz(\g)_\oo$ is a direct summand and
it does no harm to assume that $\fz(\g) \subeq \g_\oo$.

\begin{thm}{\rm(Structure Theorem for reductive Lie superalgebras)}
\mlabel{thm:3.8}
Let $\g$ be a reductive Lie superalgebra generated by $\g_\oo$
with $\fz(\g) \subeq \g_\oe$.
Then $\g \cong [\g,\g] + \fb$ is a direct sum of vector spaces
and the commutator algebra can be written as
a sum of ideals of $\g$ as follows:
\[ [\g,\g]  = \fz(\g) +  \sum_{j \in J_s} \g(j) +  \sum_{k \in J_a} \fc(k). \]
Moreover, the summation map
\[ S \: \bigoplus_{j \in J_s} \g(j) \oplus  \bigoplus_{k \in J_a}  \fc(k) \to \g \]
has central kernel.
\end{thm}

\proof It is clear that the ideals
$\g(j)$ and $\fc(k)$ are contained in the commutator algebra
$[\g,\g]$ and $[\g,\g]_\oo = \fa$, as well as
\[ [\g,\g]_\oe = \g_\oe = [\g_\oo, \g_\oo]
= [\fa,\fa] + [\fa,\fb] +[\fb,\fb].\]
As $[\fb,\fb]$ is central,
$[\fb,\fa_j] \subeq \g(j)$ for $j \in J_s$
(Lemma~\ref{teh6}(i)),
$[\fb,\fa_k] \subeq \fc(k)$ for $k \in J_a$,
and
\[ [\fa,\fa] \subeq \fz(\g) + \sum_{j \in J_s} [\fa_j, \fa_{j'}]\]
(Lemma~\ref{teh4} and Lemma~\ref{teh6}(ii)),
we thus obtain
\[ [\g,\g]  = \fz(\g)_\oe +  \sum_{j \in J_s} \g(j) +  \sum_{k \in J_a} \fc(k). \]

Now let $x_k \in \fc(k)$ and $y_j \in \g(j)$.
If $j \not=j'$, we assume that either $y_j$ or $y_{j'}$ vanishes.
Suppose that
$z := \sum_{k \in J_a} x_k + \sum_{j \in J_s} y_j \in \fz(\g)$.
Since the sum of the subspaces $\fa_j$ and $\fa_k$ of $\fa$ is direct,
all summands $x_k$ and $y_j$ are even and so they commute with $\fb$.
For $k \in J_a$, the element $z - x_k$ commutes with $\fa_k$, so that
$x_k$ does as well. As it also acts trivial on a complement of $\fa_k$ in $\fa$,
it commutes with $\g_\oo$, and thus $x_k \in \fz(\g)$ because $\g$ is generated by
$\g_\oo$. Likewise $z - y_j$ acts trivially on $\fa_j + \fa_{j'}$, hence
$y_j$ does likewise and with the same argument as before we obtain $y_j \in \fz(\g)$.
\qed

\begin{rem}\mlabel{new-late} Under the assumption that
$\g_\oe = [\g_\oo, \g_\oo]$, the Lie superalgebra $\g$ is perfect {(that is $\g=[\g,\g]$)} if and only if
$\g_\oo = [\g_\oe, \g_\oo] = \fa$, i.e., if $\fb = \{0\}$.
If this is the case, then $J = J_s$, so that $\g$ is a central quotient
of a direct sum of the ideals $\g(j)$, $j \in J_s$.
\end{rem}

\begin{rem} \mlabel{rem:a}
For the even part, we obtain with Lemma~\ref{teh6}
\[ [\g,\g]_\oe  = \fz(\g)_\oe +  \sum_{j \in J_s} [\fa_j, \fa_{j'}]
+  \sum_{k \in J_a} [\fb, \fa_k]. \]
Here the ideals $[\fb, \fa_k]$ of $\g_\oe$
are simple and act non-trivially only on the
subspace $\fa_k$ of $\fa$, whereas the ideals
$[\fa_j, \fa_{j'}]$ are reductive and act non-trivially only on the subspace
$\fa_j + \fa_{j'}$ of $\fa$.
\end{rem}

\begin{lem} Let $\g$ be a reductive Lie superalgebra with
$\g_\oe = [\g_\oo, \g_\oo]$.
Let $\fz := \fz_\fb(\fa)$ be the centralizer
of $\fa$ in $\fb$ and $\fb_r\subeq \fb$ a complementary subspace for $\fz$.
Then $\g_r := [\g,\g] + \fb_r$ is a Lie superalgebra and
$\g/\z(\g)_\oe \cong \g_r/\fz(\g)_\oe \oplus \oline\fz$
is a direct sum of Lie superalgebras, where the image $\oline\fz$ of
$\fz$ in $\g/\z(\g)_\oe$ is a central ideal and $\g_r/\fz(\g)_\oe$ is semisimple.
\end{lem}

\proof The image $\oline\fz$ of $\fz$ in $\g/\fz(\g)_\oe$ is an abelian
ideal because $[\fz,\g_\oo] = [\fz,\fb] \subeq \fz(\g)_\oe$
and $[\g_\oe, \fz] = \{0\}$. We conclude that $\g/\fz(\g)_\oe$ is a
direct sum of $\g_r/\fz(\g)_\oe$ and $\oline\fz$.

It remains to show that $\g_r/\fz(\g)_\oe$ is semisimple, i.e.,
that every solvable ideal $\fr \trile \g_r$ is contained in $\fz(\g)_\oe$.
First we observe that $\fr_\oe$ is a solvable ideal of the reductive Lie algebra
$\g_\oe$, hence central in $\g_\oe$.
Since the subspaces $\fa_k$, $k \in J_a$, and
$\fa_j + \fa_{j'}$, $j \in J_s$, are pairwise non-equivalent $\g_\oe$-modules
by Remark~\ref{rem:a}, the subspace $\fr_\oo \subeq \g_{r,\oo}$ is adapted to the
decomposition $\g_{r,\oo} =
\fb_r + \sum_{k \in J_a} \fa_k + \sum_{j \in J_s} \fa_j + \fa_{j'}$.
If $k \in J_a$ and $\fr_\oo \cap \fa_k \not=\{0\}$,
then $\fa_k \subeq \fr$, and the relation $[\fb_r, \fa_k] = [\fb,\fa_k] \subeq \fr$
contradicts the solvability of $\fr$.
If $j \in J_s$,
then $\fr \cap \g(j)$ is a solvable ideal of $\g(j)$, hence central
(Proposition~\ref{prop:3.5}(iii)),
so that $\fr \cap (\fa_j + \fa_{j'}) = \{0\}$. This implies that
$\fr_\oo \subeq \fb_r$.
Now let $b \in \fr_\oo$.
If $[b,\fa_k] \not=\{0\}$ for some
$k \in J_a$, the simplicity of the ideal $[b,\fa_k]$ of $\g_\oe$
{contradicts} $\fr_\oe \subeq \fz(\g_\oe)$.
If $[b,\fa_j] \not=\{0\}$ for some
$j \in J_s$, then
$[b,\fa_j] \subeq [\g_\oe, \g_\oe]$
(Lemma~\ref{teh6}(i)) shows that $[b,\fa_j]$ is semisimple,
which also contradicts the solvability of $\fr$.
We conclude that $[\fr_\oo, \fa] = \{0\}$. Hence
$\fr_\oo \subeq \fz \cap \fb_r = \{0\}$ and thus
$\fr = \fr_\oe \subeq \g_\oe$. As $\fr$ is an ideal, we derive that
$[\fr,\g_\oo] \subeq \fr \cap \g_\oo = \{0\}$, and since $\g_\oo$ generates
$\g$, {we see that $\fr=\fr_{\bar0}$ is central in $\g$, so $\fr\subeq\fz(\g)_{\bar0}$.}
We conclude that every solvable ideal in $\g_r/\fz(\g)_\oe$ is trivial.
\qed

\begin{rem} Let $\Lambda_n(\K)$ be the Gra\ss{}mann algebra of degree $n$ over
$\K$ and ${\mathbf W}_n(\K) := \der(\Lambda_n(\K))$ for the Lie superalgebra of
its derivations.
For a semisimple Lie superalgebra $\fs$, a slight modification
of the arguments in \cite{Ch95} shows that there exist
real simple Lie superalgebras $\fs_1,\ldots,\fs_k$
and nonnegative integers $n_1,\ldots,n_k$ such that
\[
\displaystyle
\bigoplus_{i=1}^k\big(\fs_i\otimes_{\mathbb K_i}\Lambda_{n_i}({\mathbb K_i})\big)
\subseteq
\fs
\subseteq
\bigoplus_{i=1}^k
\big(\der_{\mathbb K_i}(\fs_i)
\otimes_{\mathbb K_i}
\Lambda_{n_i}({\mathbb K_i}) +
\mathbb L_i\otimes_{\mathbb K_i} \mathbf W_{n_i}({\mathbb K_i})
\big)
\]
where, for $1 \leq i \le k$,
$\mathbb L_i$ is the centroid of $\fs_i$ and
$\mathbb K_i\subeq \bL_i$ is the subfield killed by $\der(\fs_i)$.

If, in addition, $\fs$ is reductive, then the Lie algebras
$(\fs_i)_\oe \otimes \Lambda_{n_i}({\mathbb K_i})_\oe$ must be reductive,
which immediately implies that $n_i \leq 1$ (cf.\ \cite{El96}).
The information contained in our Structure Theorem~\ref{thm:3.8}
is much finer because we obtain
\[ \g_r/\z(\g)_\oe = \g_r/\z(\g_r) \cong
\bigoplus_{j \in J_s} \g(j)/\fz(\g(j)) \oplus  \bigoplus_{k \in J_a}  T\fk_k =
\bigoplus_{j \in J_s} \g(j)/\fz(\g(j)) \oplus  \bigoplus_{k \in J_a}  \fk_k
\otimes \Lambda_1(\K), \]
where $\g(j)/\z(\g(j))$ is simple.
\end{rem}

\section{The structure of  unitary Lie superalgebras}
\mlabel{sec:4}

We now show how the decomposition derived in the previous section
can be refined to unitary Lie superalgebras.
The following lemma helps us to understand, for a compact
Lie superalgebra, the structure of $\g(j)$
in terms of the simple adjoint quotient (Proposition~\ref{prop:3.5}).

\begin{lem} \mlabel{lem:centext}
Let $\fk$ be a compact simple Lie superalgebra
and $\hat\fk$ be a central extension
with $\fz(\hat\fk) \subeq [\hat\fk,\hat\fk]$.
Then $\dim \fz(\hat\fk) \leq 1$.

More specifically, for
{$\fk=\su(n|m;\C)$}, $n \not=m$, and $\fc(n)$, all central
extensions are trivial,  for $\fk = \psu(n|n;\C)$, the Lie superalgebra
$\su(n|n;\C)$ is the universal central extension, and $\fq(n)$
is the universal central extension of $\pq(n)$.
\end{lem}

\proof  Write $\fz := \fz(\hat\fk)
= \fz_\oe \oplus \fz_\oo$. Then $\fz$ acts trivially on $\hat\fk$,
so that the adjoint action of $\hat\fk_\oe$ factors through an action
of $\fk_\oe$. Our assumption that
$\fz(\hat\fk)$ is contained in $[\hat\fk,\hat\fk]$ implies that
$\hat\fk= [\hat\fk,\hat\fk]$, so that the bracket map of
$\hat\fk$ factors through a surjective linear map
\[  \fk \otimes \fk \to \hat\fk.\]
Since $\fk$ is a semisimple $\fk_\oe$-module, so is $\fk \otimes \fk$,
and therefore $\hat\fk$ also is  a semisimple $\fk_\oe$-module.
We conclude that there exists an $\fk_\oe$-invariant
complement of the center $\fz$. We may thus write
\[ \hat\fk = \fz \oplus_\omega \fk \quad \mbox{ with } \quad
[(z,x), (z',x')] = (\omega(x,x'), [x,x']), \]
where $\omega \: \fk \times \fk \to \fz$ is a $2$-cocycle.
Since the corresponding decomposition of $\hat\fk$ is
$\fk_\oe$-invariant, $\omega(\fk_\oe, \fk) = \{0\}$.
Therefore $\omega$ is of the form
\[ \omega(x_0 + x_1, y_0 + y_1) = \beta(x_1, y_1),\]
where $\beta \: \fk_\oo \times \fk_\oo \to \fz_\oe$ is an $\fk_\oe$-invariant
symmetric bilinear map. In particular, the odd component $\beta_\oo$ vanishes.

Since the $\fk_\oe$-module $\fk_\oo$ is simple, the space
$S^2(\fk_\oo)^{\fk_\oe}$ is one-dimensional. In fact,
if $\la \cdot,\cdot\ra$ is an $\fk_\oe$-invariant positive definite
scalar product on $\fk_\oo$, then any other invariant symmetric bilinear
form $\gamma$ can be written as $\gamma(x,y) = \la Dx,y\ra$
for a symmetric $\fk_\oe$-intertwining operator $D$. Then the diagonalizability
of $D$ implies that $D \in \R\1$, so that $\gamma$ is a multiple of the
scalar product, and thus $S^2(\fk_\oo)^{\fk_\oe}$ is one-dimensional.
Now the perfectness of $\hat\fk$ implies that $\dim \fz \leq 1$.

For $\fk \cong \su(n|m;\C)$, $n \not=m$, and $\fc(n)$, all central
extensions are trivial because their complexification has this property
by Remark~\ref{rem:1.2}(c). In view of $\dim \fz \leq 1$,
$\su(n|n;\C)$ must be the universal central extension
of $\psu(n|n;\C)$, and likewise
$\fq(n)$ is the universal central extension of $\pq(n)$.
\qed

\begin{thm} \mlabel{thm:3.7}
If $\g$ is unitary and $j \in J_s$, then the ideal $\g(j)$ is isomorphic to one
of the following Lie superalgebras
\[ \su(n|m;\C), \ n \geq m \geq 1, \quad
\fq(n), \ n\geq 2, \quad \mbox{  or } \quad
\fc(n), \ n \geq 2.\]

In addition, we have:
\begin{itemize}
\item[\rm(a)] If $\g(j) \cong \su(n|m;\C), n \not=m$, or $\g(j) \cong \fc(n)$, then it is a
direct summand of $\g$.
\item[\rm(b)] If $\g(j) \cong \su(n|n;\C)$, then $[\fb,\g(j)] = \{0\}$.
\item[\rm(c)] If $\g(j) \cong \fq(n)$ and $[\fb,\fa_j] \not=\{0\}$, then
we obtain an embedding $\hat\fq(n) \into \g$.
\end{itemize}
\end{thm}

\proof We consider the compact simple Lie superalgebra $\fk := \g(j)/\z(\g(j))$
(Proposition~\ref{prop:3.5}(iii)).
In view of Theorem~\ref{thm:1.4}, it is isomorphic to one of the following
\[ \su(n|m;\C), \ n >  m \geq 1, \quad \psu(n|n;\C), \quad
\pq(n), \ n\geq 2, \quad \mbox{  or } \quad
\fc(n), \ n \geq 2.\]
For $\su(n|m;\C), n > m,$ and $\fc(n)$, all central extensions are trivial
by Lemma~\ref{lem:centext}. Since
$\g(j)$ is perfect, we obtain in this case $\fk = \g(j)$.

Lemma~\ref{lem:centext} also
implies that the simple Lie superalgebras $\psu(n|n;\C)$
and $\pq(n)$ have a unique non-trivial
central extension by a one-dimensional center.
As $\fz(\fk_\oe) = \{0\}$ in this case, the Lie superalgebra $\fk$ is not unitary.
Therefore $\g(j)$ must be isomorphic to the unique central extension $\hat\fk$.

(a) Let $\fk$ be a simple compact Lie superalgebra.
If $\fk \cong \su(n|m;\C), n \not=m,$ or $\fk \cong \fc(n)$, then all derivations
of $\fk$ are inner by Remark~\ref{rem:1.2}(c).
This implies that, in any Lie superalgebra $\g$, an ideal
isomorphic to $\fk$ is a direct summand because $\g \cong \fk \oplus \fz_\g(\fk)$.

(b) If $\fk \cong \psu(n|n;\C)$ or $\pq(n)$, then the situation is more complicated.
From Lemma~\ref{lem:centext} we know that $\fk$ has a universal central extension
isomorphic to $\su(n|n;\C)$ or $\fq(n)$, respectively. In particular,
all derivations of $\fk$ lift in a unique fashion to derivations of
$\hat\fk$, so that $\der(\hat\fk) \cong \der(\fk)$.

In $\der(\fk)$ the subspace $\ad\fk \cong \fk$ of inner derivations is an ideal,
and $\ad \fk_\oe$ is compactly embedded in $\der(\fk)$.
We therefore have an $\fk_\oe$-invariant splitting
\[ \der(\fk)_\oo \cong \ad (\fk_\oo) \oplus \fd_\oo,\]
where all derivations in $\fd_\oo$ commute with $\fk_\oe$, i.e.,
$\fk_\oe \subeq \ker D$ for $D \in \fd_\oo$. Since any such $D$ is odd, this implies that
$D^2 = 0$.

Any $D \in \der(\fk)_\oo^{\fk_\oe}$ yields an $\fk_\oe$-intertwining operator
$\fk_\oo \to \fk_\oe$. Conversely, an $\fk_\oe$-intertwining operator
$D \: \fk_\oo \to \fk_\oe$ extends to an odd derivation annihilating $\fk_\oe$ if and only if
\[ [Dx,y] = [x,Dy] \quad \mbox{ for } \quad x,y \in \fk_\oo.\]

For $\fk \cong \psu(n|n;\C)$, the space
$\fk_\oo \cong M_n(\C) \cong \C^n \otimes \oline{\C^n}$ is a simple
$\fk_\oe$-module not isomorphic to any of the two ideals
of $\fk_\oe$. Therefore $\Hom_{\fk_\oe}(\fk_\oo, \fk_\oe) = \{0\}$ implies that
$\fd_\oo = \{0\}$. Accordingly, $[\fb,\g(j)]= \{0\}$ for any ideal
of this type.

(c) For $\fk \cong \pq(n)$, we have
$\fk_\oe \cong \su(n+1,\C)$ and $\fk_\oo \cong \su(n+1,\C)$ with respect to the adjoint
action. As $\su(n+1,\C)_\C \cong \fsl(n+1,\C)$ is a simple complex Lie algebra,
the $\fk_\oe$-module $\fk_\oo$ is absolutely simple, i.e.,
$\End_{\fk_\oe}(\fk_\oo) = \R \1$. This implies that
$\dim \fd_\oo \leq 1$. We have already seen in
Example~\ref{ex:1.3}(c) that a non-zero odd outer derivation exists.
For an ideal of this type, a one-dimensional quotient of $\fb$ may act
non-trivially on $\g(j)$. For $\g(j) \cong \fq(n)$
the relation $[\fb,\fa_j] \not=\{0\}$ therefore leads to an embedding of
$\hat\fq(n) \into \g$  (cf.\ Example~\ref{ex:1.3}(c)).
\qed

\begin{ex} \mlabel{ex:tildetk} Let $\fk$ be a simple compact Lie algebra and
$T\fk = \fk \otimes \Lambda_1$,  be as in Example~\ref{ex:tk}.
We write $\tilde T\fk$ for the central extension of $T\fk$ corresponding
to a positive definite symmetric bilinear form $\kappa$ on
$\fk \cong \fk \otimes \xi$
and $\fz \cong \R$ for the corresponding center. Then
$\fh := \fz + (\fk \otimes \xi)$ is a Clifford--Heisenberg Lie superalgebra and
$\tilde T\fk \cong \fh \rtimes \fk$.
Now the faithful spin representation of $\fh$
(cf.\ Subsection~\ref{subsec:2.1}) also carries a natural
representation of $\fk$ because $\kappa$ is $\fk$-invariant.
Therefore $\tilde T \fk$ is a unitary Lie superalgebra.

If $\fh \to T\fk$ is a non-trivial central extension of Lie superalgebras
with $\fz(\fh) \subeq [\fh,\fh]$,
then the same arguments as in the proof of Lemma~\ref{lem:centext}
imply that $\fh \cong \tilde T\fk$, so that $\tilde T\fk$
is {the} universal central extension of $T\fk$. Since $T\fk$ is not unitary,
it follows that, if $\g$ is unitary and $k \in J_a$, then the ideal
$\fc(k) \trile \g$ is isomorphic to $\tilde T\fk_k$.
\end{ex}

We conclude this paper with the following theorem
which asserts that all possible types of subalgebras $\g(j)$ and $\fc(k)$
actually occur in unitary Lie superalgebras.

\begin{thm} The Lie superalgebras
\[ \su(n|m;\C), \ n \geq m \geq 1, \quad
\fq(n), \ n\geq 2, \quad \mbox{  or } \quad \fc(n), \ n \geq 2,\]
and the Lie superalgebras $\tilde T \fk$, $\fk$ compact simple, are unitary.
\end{thm}

\proof For $\tilde T\fk$, the unitarity was shown in
Example~\ref{ex:tildetk}, and for
$\su(n|m;\C)$ and $\fq(n)$, the unitarity
follows from the definition. Therefore it remains to show
that $\fc(n)$ is unitary.

To this end we use Jakobsen's classification of unitary highest weight
modules of basic classical Lie superalgebras.
In terms of \cite[Ch.~9]{Ja94} the antilinear antiinvolution
\[ \omega \: \fc(n)_\C \to \fc(n)_\C, \quad
x_0 + x_1 + i(y_0 + y_1) \mapsto
-x_0 -i x_1 -i(-y_0 -i y_1) = -x_0 - y_1 +i(-x_1 + y_0) \]
corresponds to the signs
\[ \eps_z = 1 = \eps_{p_1} = \eps_{p_2}.\]
For the corresponding real form, \cite[Props.~9.7,9.8]{Ja94} represents
a classification of the corresponding unitary highest weight modules.
Since these modules are locally finite for the subalgebra $\g_\oe$
whose real form is compact, \cite[Prop.~5.2.5]{Ka77} implies
that they are finite dimensional. We conclude that
the real Lie superalgebra $\fc(n)$ is unitary.
\qed

{\bf Acknowledgements.} {S. Azam acknowledges the support of IPM (Grant No. 93160221), and the Banach Algebra Center of Excellence for Mathematics.
He would like to thank the people of Department Mathematik, Friedrich-Alexander Universit\"at, for their hospitality during his visit.}

K.-H. Neeb acknowledges the support of DFG-grant NE 413/7-2 in the framework
of the SPP ``Representation Theory''.

\end{document}